\documentclass[10pt]{amsart}
\usepackage{amsmath,amssymb,amsfonts, amscd,enumerate}
\usepackage{amsbsy}
\usepackage{amsthm}
\usepackage{amstext}
\usepackage{amsopn}
\usepackage{marvosym}
\usepackage{graphicx}
\usepackage{latexsym}
\usepackage[T1]{fontenc}
\usepackage[latin1]{inputenc}
\usepackage{color}
\usepackage{tikz}

\newcommand{\Hmm}[1]{\leavevmode{\marginpar{\tiny%
$\hbox to 0mm{\hspace*{-0.5mm}$\leftarrow$\hss}%
\vcenter{\vrule depth 0.1mm height 0.1mm width \the\marginparwidth}%
\hbox to 0mm{\hss$\rightarrow$\hspace*{-0.5mm}}$\\\relax\raggedright #1}}}

\newtheorem{thm}{Theorem}[section]
\newtheorem{cor}[thm]{Corollary}

\newtheorem{lemma}[thm]{Lemma}
\newtheorem{pro}[thm]{Proposition}

\theoremstyle{definition}
\newtheorem*{defi}{Definition}
\newtheorem{eg}[thm]{Example}
\newtheorem{rem}[thm]{Remark}
\numberwithin{equation}{section}

\newcommand{\Z}{{\mathbb Z}}
\newcommand{\R}{{\mathbb R}}

\newcommand{\N}{{\mathbb N}}

\newcommand{\al}{{\alpha}}

\newcommand{\eps}{{\varepsilon}}
\newcommand{\gm}{{\gamma}}

\newcommand{\lm}{{\lambda}}
\newcommand{\ph}{{\varphi}}

\newcommand{\supp}{{\mathrm{supp}\,}}

\newcommand{\dx}{\,\mathrm{d}x}

\newcommand{\Hm}[1]{\leavevmode{\marginpar{\tiny%
$\hbox to 0mm{\hspace*{-0.5mm}$\leftarrow$\hss}%
\vcenter{\vrule depth 0.1mm height 0.1mm width \the\marginparwidth}%
\hbox to 0mm{\hss$\rightarrow$\hspace*{-0.5mm}}$\\\relax\raggedright
#1}}}

\begin{document}
\title[From Hardy to Rellich inequalities on graphs]
{From Hardy to Rellich inequalities on graphs}
\author[M.~Keller]{Matthias Keller}
\address{M.~Keller,  Institut f\"ur Mathematik, Universit\"at Potsdam,
	14476  Potsdam, Germany}
\email{matthias.keller@uni-potsdam.de}
\author[Y.~Pinchover]{Yehuda Pinchover}
\address{Y.~Pinchover, Department of Mathematics, Technion-Israel Institute of Technology, 3200003 Haifa, Israel}
\email{pincho@technion.ac.il}
\author[F.~Pogorzelski]{Felix Pogorzelski}
\address{F.~Pogorzelski, Institut f\"ur Mathematik, Universit\"at Leipzig, 04109 Leipzig, Germany}
\email{felix.pogorzelski@math.uni-leipzig.de}

\begin{abstract}
We show how to deduce Rellich inequalities from Hardy inequalities on infinite graphs. Specifically, the obtained Rellich inequality gives an upper bound on a function by the Laplacian of the function in terms of weighted norms. These weights involve the Hardy weight and a function which satisfies an eikonal inequality. The results are proven first for Laplacians and are extended to Schr\"odinger operators afterwards. 
	\\[2mm]
	\noindent  2000  \! {\em Mathematics  Subject  Classification.}
	Primary  \! 39A12; Secondary  31C20, 35B09, 35R02, 58E35.\\[2mm]
\end{abstract}

 \maketitle

\setcounter{section}{-1}
\section{Introduction}\label{sec-intro}
In the International Congress of Mathematicians held in 1954 in Amsterdam, Franz Rellich presented his famous inequality for smooth functions $ \ph $ of compact support in $\R^d$, $ d\neq 2 $,  which  vanish at $ 0 $ and which reads as
\begin{align*}
\int_{\R^{d}}|\Delta \ph (x)|^{2}\dx\ge \frac{d^{2}(d-4)^{2}}{16}\int_{\R^{d}}\frac{|\ph(x)|^{2}}{|x|^{4}}\dx.
\end{align*}

 The inequality was published in 1956, \cite{Rel}, after Rellich's death. Since then various versions of this inequality have been proven in various contexts. We refer to the monograph \cite{BEL15} and \cite{DH} for historical and further references.

In this paper we consider Schr\"odinger operators on infinite graphs and prove Rellich-type inequalities when given a Hardy-type inequality.  Hardy's classical inequality was originally proven in the case of the simplest graph arising from $ \N_{0} $, see \cite{KMP06}, and recently there is a rising interest in discrete and nonlocal Hardy inequalities, \cite{FS08,Gol14,KaLa16,KePiPo2}. In contrast there does not seem to be a discrete counterpart for Rellich's inequality.

Let us sketch the rough ideas and refer for details and definitions  to the next section. Given a graph over a discrete set $ X $, assume that it satisfies the following Hardy inequality with a weight $ w>0 $
\begin{align*}
\sum_{X}|\nabla \ph|^{2}\ge \sum_{X}w\ph^{2}
\end{align*}
for all compactly supported $ \ph $ (for the definition of $ |\nabla \ph|^{2} $ see the next section). Then, using Green's formula for the graph Laplacian $ \Delta $ and the Cauchy-Schwarz inequality we get
\begin{align*}
 \sum_{X}w\ph^{2}\leq \sum_{X}|\nabla \ph|^{2}=\sum_{X}\ph \Delta \ph \leq \left(\sum_{X}w\ph^{2} \right)^{\!1/2}\left(\sum_{X}\frac{1}{w}(\Delta\ph)^{2} \right)^{\!1/2},
\end{align*} 
which yields a Rellich-type inequality for all compactly supported $ \varphi $ 
\begin{align*}
\sum_{X}\frac{1}{w}(\Delta\ph)^{2} \ge\sum_{X}w\ph^{2}.
\end{align*}
An analogous computation in the continuum brings us half way to Rellich's classical inequality as it is stated above. Specifically, one gets an inequality with the weight $ |x|^{2} $ on the left hand side and the weight $ |x|^{-2} $ on the right hand side. To get a proper analogue (and, indeed, a far more flexible statement) we replace $ \ph $ by $ g^{1/2}\ph $  and prove (with somewhat more effort)  a Rellich-type inequality of the form
\begin{align*}
 \sum_{X}\frac{g}{w}(\Delta\ph)^{2} \ge(1-\gamma)\sum_{X}gw\ph^{2}
 \end{align*}
  for all compactly supported $ \varphi $, where  $0< \gamma <1 $, and $ g>0 $ is a function satisfying the eikonal inequality 
 \begin{align*}
 \frac{|\nabla g^{1/2}|^{2}}{g}\leq \gamma w\,.
 \end{align*}
  
To the best of our knowledge the connection of eikonal equalities and the study of Schr\"odinger operators goes back to Rosen \cite{Rosen}, see also \cite{Ag82,Barles,Saito}. Agmon \cite{Ag82} established for Schr\"odinger operators satisfying a Hardy inequality in $\R^d$, a Rellich-type inequality involving the Hardy weight $w$ and a positive function $g$ solving the above eikonal inequality. A Rellich inequality for Schr\"odinger operators on noncompact manifolds with optimal Hardy weights $w$ and a function $g$ satisfying the eikonal equation $|\nabla (g^{1/2})|^{2}/g=\gamma w$ was studied by Devyver/Fraas/Pinchover in \cite{DFP}. Recently, a particular case of Rellich's inequality with $g=w$ satisfying the eikonal inequality was established by Robinson \cite{Ro18} in the context of strongly local Dirichlet forms. Here, we derive such a theorem in the context of graphs as a corollary of our main result. These results are presented in Section~\ref{sec-preliminaries} and proven in Section~\ref{sec-abstract-Hardy}.

In \cite{KePiPo2} the authors showed that one obtains {\em optimal Hardy weights} associated to the Laplacian on graphs, by the supersolution construction
$$  w=\frac{|\nabla u^{1/2}|^{2}}{u}  $$
for a certain strictly positive harmonic function $ u $. In the continuum, this approach goes back  to \cite{DFP}. Given such a Hardy weight $w$ arising from a strictly positive harmonic function $ u $, we prove  in Section~\ref{sec_supersol-Constr},  a general criterion for  a function $ \theta $
such that $ g=(\theta\circ u^{1/2})^2$ satisfies the above eikonal inequality which implies the corresponding Rellich inequality.
Later in Section~\ref{sec_xalpha} we show that we can choose $ \theta $ as fractional powers which leads to the choice of
\begin{align*}
g= u^{\al}
\end{align*}
for $0< \al<1 $. This allows us to establish a discrete analogue of the classical Rellich inequality for $ \Z^{d} $, $ d\ge 5 $, see Example~\ref{ex:Zd}. Furthermore,  in Section~\ref{sec_log} we show that we can choose $ \theta $ as a logarithm  which leads to 
\begin{align*}
g=\log^{2} (u^{1/2}+1).
\end{align*}
 It seems that this weight has not been considered in the continuum so far.

By virtue of the ground state transform, we extend  the Rellich inequality from the Laplacian $ \Delta $ to  weighted Schr\"odinger operators
\begin{align*}
H=\frac{1}{m}\Delta+q,
\end{align*}
with a potential $ q $  and a measure $ m $ in Section~\ref{sec-Schr}.

Finally, in Section~\ref{sec_applications} we show how Rellich's inequality can be applied to obtain a priori estimates and existence of solutions for  the nonhomogeneous equation
\begin{align*}
Hu=f.
\end{align*}

\textbf{Acknowledgements.} The authors acknowledge the financial support of the DFG. MK and FP thank  the Technion for the hospitality where this work was done.
YP acknowledges the support of the Israel Science Foundation (grant 637/19) founded by the Israel Academy of Sciences and Humanities.

 \section{An abstract Rellich inequality}\label{sec-preliminaries}		
Let $X$ be a countably infinite set endowed with the discrete topology. We refer to the elements of $X$ as {\em vertices}.

We denote the space of real valued function on $ X $ by $C(X)$,
and denote the subspace of functions with finite  support by $ C_{c}(X) $. Via continuation by zero,   the spaces $ C(Y) $ and $ C_{c}(Y) $ are considered to be subspaces of $ C(X) $ and $ C_{c}(X) $  for subsets $ Y\subseteq X $. 
For functions $ f\in C(X) $, we denote 
\begin{align*}
\sum_{X}f:=\sum_{x\in X}f(x),
\end{align*}
whenever the right hand side converges absolutely which is obviously the case when $ f $ is in $ C_{c}(X) $. For a functions $ f $, we denote the characteristic function of the support of $ f $ by
\begin{align*}
1_{f}:=1_{\supp f}.
\end{align*}

A {\em measure} of full support on $ X $ is given by a function $ m:X\to (0,\infty) $ which extends to sets via additivity, i.e., for $A\subseteq X$
$$   m(A):=\sum_{x\in A}m(x). $$
This way we obtain the Hilbert space
$$ \ell^{2}(X,m):=\{f\in C(X)\mid \sum_{X}mf^{2}<\infty \}  $$
 with the associated norm $ \|\cdot\|_{m} $ given by $ \|f\|^{2}_m:= \sum_{X}mf^{2}$.  In the case $ m = 1 $, we denote the corresponding Hilbert space and norm by $ \ell^{2}(X) $ and $ \|\cdot\| $.

 A \emph{graph} over $X $  is a symmetric function 
 $b:X \times X\to [0, \infty)$ which vanishes at the diagonal and satisfies
 $$ \sum_{y \in X} b(x,y) < \infty,   \qquad x \in X. $$
 We  say  $x,y,\in X$ are  {\em connected}
 by an {\em edge}, and write $ x\sim y $, whenever $ b(x,y)>0 $.

The {\em formal Laplace operator} (or {\em formal Laplacian}) $ \Delta $ acts on the subspace
\[
\mathcal{F}(X):= \{ f \in C(X)\mid \sum_{y \in X} b(x,y)|f(y)| < \infty \mbox{ for all } x \in X \},
\]
 by
	\[\Delta f(x) :=  \sum_{y \in X} b(x,y)\big( f(x) - f(y) \big). 
	\]
Clearly, $ C_{c}(X)\subseteq \mathcal{F} $. Moreover, it is easy to check that $ \mathcal{F}(X) $ is an algebra with respect to pointwise multiplication. Furthermore, if $ f>0 $ is in $ \mathcal{F}(X) $, then it can be easily seen by H\"older inequality that $ f^{\al} \in \mathcal{F}(X)$ for $ \al\in (0,1) $.

We define
\begin{align*}
|\nabla f|^{2}(x):=\frac{1}{2} \sum_{y\in X}b(x,y)(f(x)-f(y))^{2},
\end{align*}
and observe that it takes finite values for $ f\in \mathcal{F}(X) $.


A fundamental concept for this paper are the so called Hardy-type inequalities. 
\begin{defi}
We say that $\Delta$ satisfies 
on $ Y\subseteq X $ the \emph{Hardy inequality}  with respect to a strictly positive \emph{Hardy weight} $ w:Y\to  (0,\infty) $ if  for all $ \varphi \in C_c(Y)  $
\begin{align*}
\sum_{X}|\nabla \ph|^{2}\ge \|\ph\|_{w}^{2}.
\end{align*}
\end{defi}
With this notation we can present one of the main results of the paper. It states that we can deduce a Rellich-type inequality from a Hardy inequality under certain assumptions on the Hardy weight.
\begin{thm}[Abstract Rellich-type inequality]\label{thm:abstractRellich}
	Let $b$ be a graph over $X$ and suppose that $w$ is a strictly positive  Hardy weight with respect to $\Delta$ on  $ X$. 
	If there is  a strictly positive function    $g \in \mathcal{F}(X)$ and $0 < \gamma < 1$ such that $g$ satisfies the  (pointwise) eikonal inequality
	\[
	 \frac{|\nabla \left(g^{1/2}\right)|^2}{g}\leq \gamma  w  \qquad \mbox{in } X ,
	\]	
	 then 	for all $  \varphi \in C_c(X) $,
	\[
	\left \|1_{\ph}   \Delta \varphi \right\|_{\frac{g}{w}} \geq  (1-\gamma)  \|   \varphi \|_{gw}.
	\]
\end{thm}
\begin{rem}
	The factor $ 1_{\ph} $ on the left hand side of the above inequality may seem technical, but it is indeed needed for applications in Section~\ref{sec_applications}. This stems from the fact that in the non-local situations of graphs, the support of $ \Delta \ph $ is in general not included in the support of $ \ph $. Here, the factor $ 1_{\ph} $ ensures that the weighted norm of $ \ph $ is estimated only by the values of $\Delta \ph  $ in the support of $ \ph $.
\end{rem}

An immediate corollary of Theorem~\ref{thm:abstractRellich} is a version of a result of Robinson \cite[Theorem~1.1]{Ro18} who proved it in the realm of strongly local Dirichlet forms.  This corollary can be seen as a direct  analogue of Rellich's original inequality \cite[Section~4]{Rel}.
		\begin{cor}[Robinson theorem for graphs]\label{core-Rob}
		Let $b$ be a  graph over $X$.
	 Assume that $w$ is a strictly positive Hardy weight with respect of $\Delta$ on $X$ such that $w \in \mathcal{F}(X)$. Furthermore, assume that there is  $0 < \gamma < 1$ such that 
		 the  following eikonal inequality is satisfied 
		\[
		 \frac{|\nabla \left(w^{1/2}\right)|^2}{w}\leq \gamma  w \qquad \mbox{in } X .
		\]		
		Then, for all $\varphi \in C_c(X)$,
		\[
		\big\| 1_{\ph} \Delta \varphi \big\| \geq (1- \gamma) \|\varphi  \|_{w^{2}}.
		\]
	\end{cor} 
	\begin{proof}[Proof of Corollary~\ref{core-Rob}]
Letting $g = w$,  the statement follows directly from Theorem~\ref{thm:abstractRellich}.
	\end{proof} 
The assumption of Robinson for strongly local Dirichlet forms is indeed a weak version of the eikonal inequality which is slightly more technical. We discuss this more general version in the next section were the main theorem is proven.
\section{Proof of the abstract Rellich-type inequality}\label{sec-abstract-Hardy}
In order to streamline the notation of the proofs we introduce some additional notation. For $f \in C(X)$, let the function $ \nabla f\in C(X\times X) $ be given by
\[
\nabla f (x,y) = f(x) - f(y) .
\]
For $ f,g\in C(X) $ let $ f\otimes g\in C(X\times X) $ be given by
\begin{align*}
(f\otimes g)(x,y)=f(x)g(y).
\end{align*}
Finally, for $ f\in C(X) $ and $ g\in C(X\times X) $ we use the convention $ fg =(f\otimes 1) g$, i.e.,
$ (fg)(x,y):=f(x)g(x,y) $,
which yields with the notation above
\begin{align*}
\sum_{X\times X}fg:=\sum_{x,y\in X}f(x)g(x,y),
\end{align*}
whenever the right hand side converges absolutely. 

The following basic identity is fundamental for the proof of Theorem~\ref{thm:abstractRellich}.

\begin{lemma} \label{lemma:main}
	For $\varphi \in C_c(X)$ and $f \in \mathcal{F}(X)$, one has
\begin{align*}
\sum_{X}(f^{2}\ph) \Delta \ph = \sum_{X}|\nabla( f\ph)|^{2}- \frac{1}{2}\sum_{X\times X}b(\ph\otimes\ph)(\nabla f)^{2}.
\end{align*}
\end{lemma} 

\begin{rem} Recall that $ (\nabla f)^{2} $ is a function in $ C(X\times X) $ given by $(\nabla f)^{2}(x,y)=(f(x)-f(y))^{2}  $, while $ |\nabla f|^{2} $ is a function in $ C(X) $ which is defined as $ |\nabla f|^{2}(x)=\frac{1}{2}\sum_{y\in X}b(x,y)(\nabla f)^{2}(x,y) $.	
\end{rem}
\begin{proof}[Proof of Lemma~\ref{lemma:main}]
Recall the {  Green's formula} from \cite[Lemma~4.7]{HK}
	for $\ph \in C_c(X)$ and $f \in \mathcal{F}(X)$
	\begin{equation*}\label{eq-green form}
	\frac{1}{2}\sum_{X\times X}b\nabla \ph \nabla f=	\sum_{X} \ph \Delta f = 	\sum_{X} f \Delta \ph ,
	\end{equation*}
	where all of the sums converge absolutely. Furthermore, by a direct calculation, we have the following {Leibniz rule} for the Laplacian
	\begin{equation*}\label{eq-Leibniz-r}
	\Delta (f\ph)(x)=f(x)(\Delta\ph )(x)+\ph(x)(\Delta f)(x)-\sum_{y\in X}b(x,y)(\nabla \ph\nabla f )(x,y).
	\end{equation*}
	So, by the Leibniz rule  and  Green's formula  we obtain
	\begin{align*}
	\sum_{X}(f^2 \varphi) \Delta \varphi &=\sum_{X}f\varphi \Delta (f\varphi)-\sum_{X}f\varphi^2  \Delta f+\sum_{X\times X}   f \varphi b\nabla \ph \nabla f\\
&=\sum_{X}|\nabla( f\ph)|^{2}-\frac{1}{2}\sum_{X\times X}b\nabla (f\ph^{2}) \nabla f+\sum_{X\times X}   f\varphi b\nabla \ph  \nabla f.
	\end{align*}
Applying the formula
	\begin{align*}
	\nabla (\ph^{2}f )(x,y)=(( f\ph)(x)+( f\ph)(y))\nabla\ph(x,y)+\ph(x)\ph(y)\nabla f(x,y)
	\end{align*}
	to the second term on the right hand side readily yields the statement.
\end{proof}
We deduce the abstract Rellich-type inequality, Theorem~\ref{thm:abstractRellich}, from a slightly more general yet more technical result, Theorem~\ref{thm:abstractRellich2} below. Observe that the weak eikonal inequality which replaces the pointwise eikonal inequality is a precise analogue of assumption (II)  in the work of Robinson \cite[Theorem~1.1]{Ro18}. 

\begin{thm}\label{thm:abstractRellich2}
	Let $b$ be a graph over $X$, and suppose that $w$ is a strictly positive  Hardy weight with respect of $\Delta$ on a subset $Y \subseteq X$. 
	If there is  a  
strictly	positive function    $g \in \mathcal{F}(Y)$ and $0 < \gamma < 1$ such that $g$ satisfies for some (respectively, all) $ \ph\in C_{c}(Y) $ the following weak eikonal inequality
	\[
\frac{1}{2}	\sum_{X\times X}b (\ph\otimes \ph)(\nabla g^{1/2})^2\leq \gamma\| \ph\|^{2}_{ gw},
	\]	
	then for such a (respectively, all) $  \varphi \in C_c(Y) $
	\[
	\left \|  1_{\ph} \Delta \varphi \right\|_{\frac{g}{w}} \geq  (1-\gamma)  \|   \varphi \|_{gw}.
	\]
\end{thm}
\begin{proof}
	We use  Lemma~\ref{lemma:main} with  $ f=g^{1/2} $ (where $ g $ is taken from the assumption) together with the Hardy inequality and the weak eikonal inequality from the theorem's assumptions, to obtain for  $\varphi \in C_c(Y)$
	\begin{align*}
	\sum_{X} (g\ph) \Delta \ph\ge \|\ph g^{1/2}\|_{w}^{2} -\frac{1}{2}\sum_{X\times X}b(\ph\otimes \ph)(\nabla g^{1/2})^{2}\ge (1-\gm) \|\ph \|_{gw }^{2}\,.
	\end{align*}
	On the other hand,  the Cauchy-Schwarz inequality gives 
	\begin{align*}
	\sum_{X}(g\ph)\Delta \ph= \sum_{X}(g\ph 1_{\ph}) \Delta \ph\leq \|1_{\ph} \Delta\ph\|_{\frac{g}{w}}\|\ph\|_{{g}{w}}\,.
	\end{align*}
	Combining these two estimates yields the statement of the theorem.	
\end{proof}
\begin{proof}[Proof of Theorem~\ref{thm:abstractRellich}]
	Assume $ g $ satisfies $  |\nabla g^{1/2}  |^{2} \leq \gamma g w $. Then, we obtain for all $ \ph\in C_{c}(X) $ by Young's inequality, $ 2 \varphi(x) \varphi(y) \leq \varphi^2(x) + \varphi^2(y) $, 
	\begin{align*}
	\frac{1}{2}\sum_{X\times X}b (\ph\otimes \ph)(\nabla g^{1/2})^2\leq \sum_{X}\ph^{2}|\nabla g^{1/2} |^{2}\le \gamma\sum_{X}   gw \ph^{2}.
	\end{align*}
	Thus, we can apply the theorem above, Theorem~\ref{thm:abstractRellich2}, and the statement follows.
\end{proof}

\begin{rem}It is obvious from the proof that it is sufficient that $ g\ge 0 $ for the statement of Theorem~\ref{thm:abstractRellich2}.
\end{rem}

To end this section we present a corollary of Theorem~\ref{thm:abstractRellich2} which allows us to restrict ourselves to subsets $ Y \subseteq X$. This will be used frequently in the subsequent sections.
\begin{cor}
	\label{cor:abstractRellich}
	Let $b$ be a graph over $X$ and suppose that $w$ is a strictly positive  Hardy weight with respect to $\Delta$ on a subset $Y\subseteq X$. 
	If there is  a strictly positive function    $g \in \mathcal{F}(Y)$ and $0 < \gamma < 1$ such that $g$ satisfies the  eikonal inequality
	\[
	\frac{1}{2}\sum_{y\in Y}b(x,y)({g^{1/2}(x)-g^{1/2}(y)})^{2} \leq \gamma  g(x)w(x),  \quad x\in  Y,
	\]	
	then, for all $  \varphi \in C_c(Y) $,
	\[
	\left \|  1_{\ph} \Delta \varphi \right\|_{\frac{g}{w}} \geq  (1-\gamma)  \|   \varphi \|_{gw}.
	\]
\end{cor}
\begin{proof}
	To check that for $ \ph\in C_{c}(Y) $ the weak eikonal inequality in the assumption of Theorem~\ref{thm:abstractRellich2} is satisfied, observe that by Young's inequality  
	\begin{align*}
	\frac{1}{2}\sum_{X\times X}b (\ph\otimes \ph)(\nabla g^{1/2})^2=&
	\frac{1}{2}\sum_{Y\times Y}b (\ph\otimes \ph)(\nabla g^{1/2})^2\\
	\leq&\frac{1}{2}\sum_{x\in  Y}\ph^{2}(x)\sum_{y\in Y}b(x,y)  ( g^{1/2}(x)-g^{1/2}(y))^2.
	\end{align*}
	Thus, the statement follows from our assumption and Theorem~\ref{thm:abstractRellich2}.
\end{proof}

\section{Rellich-type inequality via the supersolution construction}\label{sec_supersol-Constr}

In the following we use the {\em supersolution construction} which produces Hardy weights on graphs. To this end, we consider strictly positive superharmonic functions $u:X \to (0, \infty)$. 
Here, a function $u\in \mathcal{F}(X)$ is said to be {\em (super)harmonic on $Y \subseteq X$} if $ \Delta u=0 $ (respectively, $ \Delta u\ge0 $) on $ Y $.
\medskip

Let $u$ be  a strictly positive superharmonic $u$ which is not harmonic. Then the so called {\em supersolution construction} yields  a positive function
\[
w= \frac{\Delta u}{u} \, .
\]
 Thus, $ u $  is a positive solution of the equation
\begin{align*}
(\Delta-w)v=0.
\end{align*}
By the Allegretto-Piepenbrink theorem \cite[Theorem~4.2]{KePiPo1}, $ w$ is a Hardy weight for the Laplacian on $X$, i.e.,  for all $ \ph\in C_{c}(X) $
\begin{align*}
\sum_{X}|\nabla \ph|^{2}\ge \sum_{X}w \ph^{2}.
\end{align*}
It is not hard to check  that for any strictly positive  (super)harmonic $u$ and $0<\alpha<1$, we have $ \Delta u^{\alpha}\ge 0 $. Thus, we obtain the Hardy weight 
\[
w= \frac{\Delta (u^{\alpha})}{u^{\alpha}} \, ,
\]
and, especially, for $\alpha=1/2$ we have (see  \cite[Lemma~2.2]{KePiPo2})
 \[
 w = \frac{\Delta (u^{1/2})}{u^{1/2}}=\frac{1}{2u^{1/2}}\Delta u+\frac{|\nabla u^{1/2}|^{2}}{u}\,.
 \]
With some additional assumptions on $u$, one even gets optimality of this Hardy weight $w$, \cite[Theorem~1.1]{KePiPo2}. Specifically, if $ u $  is {\em proper} (i.e., the preimage of compact sets in $ (0,\sup u)\subset [0,\infty)$ is finite), harmonic outside a finite set $K$, and satisfies the \emph{anti-oscillation condition} 
\begin{equation*}
\sup_{x\sim y} \frac{u(x)}{u(y)}<\infty,
\end{equation*}
then, 
\begin{itemize}
	\item $\Delta-w $ is {\em critical}, i.e., for all Hardy weights $ w' $ with $ w'\ge w $ we have $ w=w' $.
	\item $\Delta-w$ is {\em null-critical with respect to} $w$, i.e., $u^{1/2}$, the ground state of $ \Delta-w $, is not in $ \ell^{2}(X,w) $.
\end{itemize}
We note that the above two properties imply that $ w $ is optimal at infinity, i.e., if  $ \lambda w $ is a Hardy weight on $ X\setminus K $ for a finite set $ K $ and some $\lambda > 0 $, then $ \lm\leq 1 $.
In what follows we will see some versions of the anti-oscillation condition above.
\medskip

Using the  Hardy weights obtained by the supersolution construction, we can prove a  Rellich-type inequality for functions supported in a  subgraph which arises from a strengthening of the anti-oscillation condition above. 

\medskip

Recall that the Rellich-type inequality on a subgraph $Y \subset X$, Corollary~\ref{cor:abstractRellich}, is built on the assumption that there is some $0<\gamma < 1$ and some suitable function $g \in \mathcal{F}(Y)$ with $ g>0 $ such that 
\[
\frac{|\nabla (g^{1/2})|^2}{g} \leq \gamma  w
\]
on $Y$. 
We will see in the present section that in some situations, we can work with $g = (\theta \circ u^{1/2})^2$, where $u$ is the strictly positive function used in the supersolution construction and $\theta:[0, \infty) \to [0, \infty)$ is a suitable function. This is based on the concept of admissible functions.

\begin{defi}
 Let $ c\ge 0 $. A function $\theta:[c, \infty) \to [0, \infty)$ is called 
	{\em  $\eps$-admissible} for  $0 < \varepsilon < 1$ if there is some $\gamma = \gamma( \varepsilon) \in[0, 1)$ such that 
\begin{align*}
 |\theta (t)-\theta (at)|\leq \gamma^{1/2}|1-a|\theta (t),  
\end{align*}
for all  $ t\in [c,\infty)$, $a \in [\varepsilon, \infty) $ such that $ at\ge c $.
If $\theta$ is an  $\eps$-admissible function, then we call $\gamma= \gamma(\varepsilon)$ an {\em  admissible constant} for $\varepsilon$. 
\end{defi}

\begin{rem}
	In the next sections we will show that $ t\mapsto t^{\alpha} $ and $ t\mapsto \log(t+1) $ are admissible functions for suitable intervals $ [c,\infty) $. To get some intuition, let us discuss a consequence of admissibility. Assuming that $ \theta $ is strictly positive on $[0,\infty)$, then admissibility  means that the functions
	\begin{align*}
\theta_t: [\eps,\infty)\to [0,\infty),\qquad	\theta_t(a):=\frac{\theta (at)}{\theta(t)}	
	\end{align*}
	are $ \gm^{1/2} $-Lipschitz  continuous  at $ a=1 $ for all $ t\in [0,\infty) $.
	On the other hand, if an $ \eps $-admissible function $ \theta $ satisfies $ \theta(t_{0}) =0$ for $ t_{0}> 0 $, then $ \theta\equiv 0$ on $ [\eps t_{0},\infty) $  whenever  $ \eps t_0\geq c $ and iteratively on $ (c,\infty) $.
\end{rem}

\begin{thm}[Rellich-type inequality via the supersolution construction]  \label{thm:RellichboundedGEN}
	Let $ b $ be a graph over $ X $. Let $ u $ be a strictly positive superharmonic function, and assume that the Hardy weight
	$w={\Delta(u^{1/2})}/{u^{1/2}} $ is strictly positive. Fix $0< \eps<1 $ and let $\theta: [\inf u^{1/2},\infty)\to [0,\infty)$  be   $\eps$-admissible  with constant $ \gamma=\gamma(\eps)$.
	Then, 
		\[
	{\left \|1_{\ph}  \Delta \varphi  \right\|}_{\frac{(\theta \circ {u^{1/2}})^{2}}{w}} \geq   (1-\gamma )  \|  \varphi \|_{ (\theta \circ u^{1/2} )^{2}w } 
		\]	
	for all $ \ph\in C_{c}(X_{\eps}) $, where
	\begin{eqnarray*}
		X_{\varepsilon} &:=& \left\{x \in X\mid \inf_{y \sim x} \frac{u(y)}{u(x)} \geq \varepsilon^{2} \right\}.
		\end{eqnarray*} 
\end{thm}

\begin{proof} In view of the remark above $ \theta $ is either 
	zero or strictly positive on the range of $ u^{1/2} $.
	Clearly, the statement is trivial for $ \theta \equiv 0 $ and therefore we may assume that $ \theta>0 $ on the range of $ u^{1/2}$. 	We  apply the Rellich-type inequality given by Corollary~\ref{cor:abstractRellich} with $Y= X_{\varepsilon}$,  $g^{1/2} = (\theta \circ u^{1/2})$ and $$ w =\frac{\Delta u^{1/2}}{u^{1/2}}=\frac{1}{2u^{1/2}}\Delta u+\frac{|\nabla u^{1/2}|^{2}}{u} \ge\frac{|\nabla u^{1/2}|^{2}}{ u} $$ on $ X_{\eps} $. Plugging  $ t=u^{1/2} $ into the assumption of admissibility of $ \theta $ yields that
		\begin{align*}
		\frac{	|\theta( u^{1/2}(x))-\theta( u^{1/2}(y))|}{\theta( u^{1/2}(x))}\leq \gamma^{1/2} 
		\frac{	| u^{1/2}(x)- u^{1/2}(y)|}{u(x)^{1/2}}\, ,
		\end{align*}
		whenever $ u(y)/u(x)\ge \eps^{2} $.
		Hence,  we infer
		\begin{align*}
		\frac{|\nabla (\theta\circ u^{1/2})|^{2}}{(\theta\circ u^{1/2})^{2}}\leq 	\gamma\frac{|\nabla  u^{1/2}|^{2}}{u}\le\gamma w
		\end{align*}
		on $ X_{\eps} $.
		Thus, the function $ g=(\theta\circ u^{1/2})^{2} $ satisfies the eikonal inequality in Corollary~\ref{cor:abstractRellich} and we conclude the result.
\end{proof}

A graph is said to have {\em standard weights} if $ b $ maps $ X\times X $ into $ \{0,1\} $. This is the setting of combinatorial graphs where all edges have the weight $ 1 $ and the vertex degree is the number of edges emanating from a vertex. Note that the assumption $ \sum_{y\in X}b(x,y)<\infty $ for all $ x\in X $ implies that graphs with standard weights are locally finite, i.e., there is a finite number of vertices emanating from each vertex.

\begin{thm} \label{thm:RellichGENGLOB}
	Let $b$ be a  graph over $X$ with standard weights and  vertex degree bounded by $ D $. 
 Let $ u $ be a function which is strictly positive and  superharmonic  on $ X\setminus K $ and vanishes on a set $ K \subset X$. Assume that the Hardy weight
 $w={\Delta(u^{1/2})}/{u^{1/2}} $ is strictly positive on $ X\setminus K $. 	Suppose that there is an $\eps$-admissible function $\theta:[\inf u^{1/2}, \infty) \to [0,\infty)$   for $ \eps=D^{-1/2} $, and let $0\le \gamma<1 $ be the corresponding admissibility constant. Then, for all $  \varphi \in C_c(X\setminus K) $,
	\[
	{\left \|1_{\ph}  \Delta \varphi  \right\|}_{\frac{(\theta \circ {u^{1/2}})^{2}}{w}} \geq   (1-\gamma )  \|  \varphi \|_{ (\theta \circ u^{1/2} )^{2}w }.
	\]	
\end{thm}
\begin{proof}
Let $ D $ be the upper bound on the vertex degree and $ Y=X\setminus K $. It is easy to verify that for a strictly positive superharmonic function and $ x,y\in Y $ such that $ x\sim y $, we have the local Harnack inequality 
\begin{align*}
\frac{u(y)}{u(x)}\ge \frac{1}{D}\,.
\end{align*} 
Hence, admissibility yields that for all $ x,y\in Y $ such that $ x\sim y $ we have
\begin{align*}
\frac{(\nabla( \theta\circ u^{\frac{1}{2}}))^{2}(x,y) }{\theta( u^{\frac{1}{2}}(x))^{2}}\leq \gamma\frac{(\nabla(  u^{\frac{1}{2}}))^{2}(x,y) }{ u(x)}\, .
\end{align*}
Summing about the neighbors of $ x $ in $ Y $ yields
\begin{align*}
\frac{1}{2\left(\theta(u^{\frac{1}{2}}(x))\right)^2}\sum_{y\in Y}b(x,y) ( \theta(u^{\frac{1}{2}}(x))-\theta(u^{\frac{1}{2}}(y)))^{2}
\leq\gamma\frac{|\nabla \left(u^{1/2(x)}\right)|^{2}}{u(x)}\leq \gamma w(x).
\end{align*}
Thus, the statement follows from Corollary~\ref{cor:abstractRellich} applied with $ g= (\theta\circ u^{\frac{1}{2}})^{2}$.
\end{proof}
\begin{rem}
Suppose that $X$ is a weighted graph satisfying the {\em uniform ellipticity condition}: There exists $\lambda>0$ such that for every $x,z\in X$ such that $x\sim z$ we have  
$$\sum_{y\in X}b(x,y)\leq \lambda b(x,z). $$
By the same means as above, we get the assertion of Theorem~\ref{thm:RellichGENGLOB} in the case of weighted graphs with $D$ replaced by $\lambda$.
Indeed, the above ellipticity condition clearly implies the {\em uniform Harnack inequality}:
for any strictly positive superharmonic function $u$ in $X$ and $ x,y\in X $ 
such that $ x\sim y$ we have
\begin{align*}
\frac{u(y)}{u(x)}\geq \frac{1}{\lambda} \,.
\end{align*}
Hence, the proof of Theorem~\ref{thm:RellichGENGLOB} can be applied verbatim.
\end{rem}

\section{The function $\theta(t) = t^{\alpha}$} \label{sec_xalpha}
In the present section we show that for $0 < \alpha < 1$, the function $\theta:[0, \infty) \to [0, \infty), \,\, \theta(t) = t^{\alpha}$ is $\eps$-admissible for $0<\eps<1$. Therefore, in view of Theorem~\ref{thm:RellichGENGLOB}, we get the following corollary. 

\begin{cor} \label{thm:Rellichbounded} 	Let $b$ be a  graph over $X$ with standard weights and  vertex degree bounded by $ D $.  Let $ u $ be a function which vanishes on a set $ K \subset X$ and is  strictly positive and  superharmonic  on $ X\setminus K $. Assume that the Hardy weight
	$w={\Delta(u^{1/2})}/{u^{1/2}} $ is strictly positive on $ X\setminus K $.
	Then, for all $0 <\alpha < 1$ and $ \ph\in C_{c}(X\setminus K) $,
	\[
	\Big \|1_{\ph}  \Delta \varphi \Big\|_{ \frac{u^{\alpha}}{w}} \geq (1-\gamma)
	 \|   \varphi \|_{wu^{\alpha} } ,
	\]
where 
	\begin{align*}
		\gamma =  \left( \frac{1- D^{-\alpha/2}}{1 - D^{-1/2} } \right)^2.
	\end{align*}
\end{cor}

The proof of the above corollary rests on the following lemma. 

\begin{lemma} \label{lemma:sqdominatedxalpha}
	For $0 < \alpha < 1$, the function $\theta:[0, \infty) \to [0, \infty), \,\, \theta(t) = t^{\alpha}$ is $\eps$-admissible  for $0 < \varepsilon <  1$ with  admissible constant
	\[
	\gamma(\varepsilon) = \left( \frac{1 - \varepsilon^{\alpha}}{1- \varepsilon} \right)^2.
	\]	

\end{lemma} 
\begin{proof}
	Fix $0 < \alpha < 1$.
	We define the function  $ 	\vartheta: [\varepsilon, \infty) \to \R $ for $ t\neq 1 $ as
	\[
 \vartheta(t) = 
	\Big(\frac{1- t^{\alpha}}{1- t} \Big)^2,
	\]
	and $ \vartheta (1)=\al^{2} $.
	Using L'H\^opital's rule, we see that  $\vartheta$ is a continuously  differentiable  function on $[\varepsilon, \infty)$ whose derivative for $ t\neq 1 $ is given by 
	\[
	\vartheta^{\prime}(t) =
	\frac{2(t^{\alpha} -1 )}{(t-1)^3}  \big( 1 - t^{\alpha} - \alpha t^{\alpha - 1} (1 - t) \big),
	\]
	and $ \vartheta^{\prime} (1) =\alpha^2 (\alpha - 1)$.
	We claim that $\vartheta^{\prime}(t) < 0$ for all $t \in [\varepsilon, \infty)$. Since $\alpha < 1$, this is clearly true for $t = 1$. For $t >  1$, we use the mean value theorem in order to find some $\xi \in (1, t)$ such that 
	\[
	(1^{\alpha} - t^{\alpha}) = \alpha \xi^{\alpha -1} (1-t). 
	\]
	Now, since $t > 1$, $\xi < t$ and $\alpha < 1$, this yields
	\[
	1 - t^{\alpha} - \alpha t^{\alpha - 1}(1-t) < 0,
	\]
	which clearly shows that $\vartheta{'}(t) < 0$. 
For $t < 1$, we proceed in the same way in order to obtain $\vartheta{'}(t) < 0$ in this case as well. 
	
	In conclusion, we have shown that $\vartheta$ is strictly monotonically decreasing. This shows that for all $ t \in [\varepsilon, \infty) $ the inequality 
	\[
	(1 - t^{\alpha})^2 \leq \gamma(\alpha, \varepsilon)  (1 - t)^2
	\]
	is satisfied for 
	\[
	\gamma(\alpha, \varepsilon) := \vartheta(\varepsilon) = \left( \frac{1- \varepsilon^{\alpha}}{1- \varepsilon} \right)^2.
	\]
	The observation that $\vartheta(\varepsilon) < \vartheta(0) = 1$ finishes the proof. 
\end{proof}
	\begin{proof}[Proof of Corollary~\ref{thm:Rellichbounded}]	It follows from Lemma~\ref{lemma:sqdominatedxalpha} that $\theta$ given by $ \theta(t)=t^{\alpha} $ is $\eps$-admissible for $0<\eps < 1$. Hence, we can apply Theorem~\ref{thm:RellichGENGLOB} in order to obtain the claimed Rellich-type inequality. Note that Lemma~\ref{lemma:sqdominatedxalpha} also shows that $\gamma$ is an  admissible constant for $\varepsilon = D^{-1/2}$. 
\end{proof}
We apply the above findings to the line graph, as well as to the standard graph on $\Z^d$, where $d \geq 3$.
\begin{eg} \label{eg:linegraphroot}
	We consider the  line graph on $\N_{ 0}:=\N\cup \{0\}$ with standard weights where $k \sim l$ if and only if $|k-l|= 1$. Then the function $n:\N_{ 0} \to \N_{ 0}$ given by  $n(k) = k$ is harmonic in $\N$. We apply Corollary~\ref{thm:Rellichbounded} with  $K=\{0\}$ and  $ D=2 $. Therefore, for a given $0 < \alpha < 1$, we get the validity of a Rellich-type inequality. Let us explicate all components of this inequality. 
	Firstly, the   Hardy weight   
	$$w:\N \to (0,\infty), \qquad w=\frac{\Delta(n^{1/2})}{n^{1/2}} $$
	arising from the supersolution construction of $ n $ is optimal and can be explicitly computed (confer \cite{KePiPo3}) to be 
	\begin{align*}
	w(k)=2-\left(1+\frac{1}{k}\right)^{1/2}-\left(1-\frac{1}{k}\right)^{1/2}=
{\sum_{l=1}^{\infty} \binom{4l}{2l} \frac{1}{(4l-1)2^{4l-1}} \frac{1}{k^{2l}}} > \frac{1}{4k^{2}}
	\end{align*}
	for $ k\ge2 $, and $ w(1)=
	2 - \sqrt{2} >1/4$.
	Furthermore,  Corollary~\ref{thm:Rellichbounded}  provides a  constant for the Rellich-type inequality by
	\begin{align*}
		\gamma = \Bigg( \frac{1- D^{-\alpha/2}}{1 - D^{-1/2} } \Bigg)^2 = \Bigg( \frac{1- 2^{-\alpha/2}}{1 - 2^{-1/2} } \Bigg)^2.
	\end{align*}
	It turns out that for any $0<\alpha<1$, we have $0<\gamma <1$. We finally obtain the Rellich-type inequality
	\[
	\| 1_{\ph}\Delta \varphi\|_{ \frac{n^{\alpha}}{w}} \geq (1-\gamma) \| \varphi\|_{w {n^{\alpha}} }
	\]
	for all $\varphi\in C_{c}(\N)$ which extend to $ \N_{0} $ by $\varphi(0)= 0$, where $n$, $ \gamma $ and $w$ are given above.

	Using the basic estimate $ w(k)\ge (2k)^{-2} $, we get for any $0\leq \alpha <1$
	 $$  \sum_{k=1}^{N}k^{\al+2}(\Delta \ph)^{2}(k)\ge \frac{(1-\gamma)^{2}}{16} \sum_{k=1}^{N}k^{\al-2}\ph^{2}		(k), $$
	 where $ \supp\ph\subseteq\{1,\ldots,N\}$. By the virtue of the mean value theorem one can further estimate $ \al ^{2}<\gamma < \al^{2}2^{1-\al}< 1 $. In contrast, the classical case of the Dirichlet Laplacian on $ (0,\infty) $ yields a constant $ \gamma =\al^{2} $, confer \cite[Example~10.4]{DFP}.
\end{eg}
Next, we give an explicit Rellich-type inequality for the standard graph on $ \N^{d}_{0}\subseteq  \Z^d$, where  $d \geq  2$.
\begin{eg}\label{eg:quadrant} For  $d\geq 2$, consider the standard  subgraph  $ \N^{d}_{0} $ of $ \Z^{d} $,  and let $ K:=\Z^{d}\setminus\N^{d} $. Then, the function $ x: \N^{d}_{0}\to[0,\infty) $ given by $ x(k)=k_{1}\cdot\ldots\cdot k_{d} $  is harmonic in $ \N^{d} $  and vanishes on $ K $. The supersolution construction yields a Hardy weight  $ w_{d}:\N^{d}\to (0,\infty) $
	$$ w_{d}(k): =\frac{\Delta(x(k)^{1/2})}{x(k)^{1/2}} =w(k_{1})+\ldots+ w(k_{d}) > \frac{1}{4}\left(\frac{1}{k_{1}^2}+\ldots+\frac{1}{k_{d}^2}\right),  $$
	where $ w $ is the Hardy weight on $ \N $ from Example~\ref{eg:linegraphroot} above. Since the vertex degree is $ 2d $, we obtain a Rellich inequality for $ 0<\al<1 $ and all $ \varphi \in C_c(\N^d) $
	\begin{align*}
	\left\|1_{\ph}\Delta \ph\right\|_{\frac{ x^{\al}}{w_{d}}}\ge \left(1-\Bigg( \frac{1- (2d)^{-\alpha/2}}{1 - (2d)^{-1/2} } \Bigg)^2\right)	\left\| \ph\right\|_{x^{\al}w_{d}} .
	\end{align*}
	For $ {k}_{j}\to\infty $  (leaving all other coordinates fixed)  the asymptotics of the weights on both sides of the inequality are $ {k}_{j} ^{\al}$. On the other hand, letting $ k_{1}=\ldots= k_{d}=t $ the asymptotics are $ {4}t^{\al d+2}/{d} $ on the left hand side and $ {d}t^{\al d-2}/{4} $ on the right hand side. 
\end{eg}
We conclude the section with an explicit Rellich-type inequality for the standard graph on the lattice $X = \Z^d$, where $d \geq 3$.
\begin{eg}\label{ex:Zd}
	We consider the standard graph on the lattice $X = \Z^d$, for $d \geq 3$, where two elements $k,l \in \Z^d$ are connected via an edge if and only if
	$| k-l | = 1$.  Let $G:X\times X\to (0,\infty)$ be the (positive minimal) Green function of the (standard) Laplacian on $X$  which is given by 
	\begin{align*}
	G(k,l):=\sum_{n=0}^{\infty}p_{n}(k,l), \qquad k,l\in \Z^{d},
	\end{align*}
	where $p_{n}(k,l)$ are the matrix elements of the $n$-th power of
	the transition matrix given by the matrix elements
	$p_{1}(k,l)=1/2d$ if $k\sim l$ and $0$ otherwise.
	Denote $G(k) := G(k,0)$.
	By the supersolution construction it follows that
	\[
	w = \frac{\Delta G^{1/2}}{G^{1/2}}
	\]
	gives rise to a strictly positive  Hardy weight for the Laplacian (which is in fact optimal, \cite{KePiPo2}). 
	We derive the Rellich-type inequality, Corollary~\ref{thm:Rellichbounded}, for all $ \varphi \in C_c(\Z^d) $
	\[
	\Big\| 1_{\supp \ph}\Delta \varphi  \Big\|_{\frac{G^{\alpha}}{w}} \geq \left( 1 - \left( \frac{1 - (2d)^{-\alpha/2}}{1-(2d)^{-1/2}} \right)^2 \right)  \| \varphi\|_{G^{\alpha}w} ,
	\]
	where $0<\alpha<1$. 
	
	For the reader who is interested in the asymptotics, we mention the corresponding asymptotics of $ w $ and $ G .$ Specifically,  \cite[Theorem~7.2]{KePiPo2} shows that $w$ has the
	asymptotic behavior 
	\[
	w(k) = \frac{(d-2)^2}{4}  \frac{1}{|k|^2} + \mathcal{O}\big( |k|^{-3} \big) \qquad \mbox{as } |k| \to \infty.
	\]
	 Furthermore, by \cite[Theorem~2]{Uch98} one has
	\[
	G(k) = \frac{C_1(d)}{|k|^{d-2}} + C_2(d)\left(\left ( {\sum_{i=1}^d \Big(\frac{k_i}{|k|} \Big)^4}\right) - \frac{3}{d+2}\right )\frac{1}{|k|^d}
	+ \mathcal{O}\left( \frac{1}{|k|^{d+2}} \right),
	\]
	where $C_1(d)$ and $C_2(d)$ are positive constants depending only on $d$. For $ d\ge 5 $, we  recover  the correct asymptotics of the weights with respect to the classical Rellich inequality. Specifically, we can choose $ \al=2/(d-2)<1 $ for $ d\ge 5 $ and get $$  G^{\al}/w \asymp C \quad\mbox{and}\quad G^{\al}w \asymp C'\frac{1}{|k|^{4}}  $$ as $    |k|\to \infty $ with positive constants $ C $ and $ C' $.
\end{eg}

\section{The function $\theta(t) = \log(t+1)$}\label{sec_log} 
In the present section we show that for all $c > 0$ and $0<\eps<1$, the function
$$\theta:[c, \infty) \to [0, \infty), \quad  \theta(t) := \log(t+1)$$ 
is $\eps$-admissible. As a consequence, one gets a Rellich-type inequality via the supersolution construction if the corresponding positive supersolution $u$ is bounded away from zero. This leads to the following corollary of Theorem~\ref{thm:RellichGENGLOB}.

\begin{cor} \label{thm:Rellichboundedlog}
	Let $b$ be a  graph over $X$ with standard weights, and  vertex degree bounded by $ D $.  Let $ u $ be a  strictly positive superharmonic which is bounded from below by $ c>0 $ on 
	$ X\setminus K $, and which vanishes on a set $ K \subset X$. Assume that the Hardy weight 
	$w={\Delta(u^{1/2})}/{u^{1/2}} $
	is strictly positive on $ X\setminus K $.
	Then for all $ \ph\in C_{c}(X\setminus K) $
	\[
	\Big \|1_{\ph}  \Delta \varphi \Big\|_{ \frac{\log^{2} (u^{1/2}+1)}{w}} \geq (1-\gamma)  \|   \varphi \|_{w\log^{2} (u^{1/2}+1) }  \,,
	\]
	where $ 0<\gamma<1 $ is given by
	\begin{align*}
	\gamma = \frac{\left( 1 - \frac{ \log(c^{1/2}  D^{-1/2} + 1)}{\log(c^{1/2} + 1)} \right)^2}{\big(1- D^{-1/2}\big)^2}.
	\end{align*}
\end{cor}

\begin{rem}Let $ u $ be a strictly positive superharmonic function.
	
	Note that the function  $f:t\mapsto\log(t^{1/2}+1)$ used in the theorem above is  concave and increasing on $(0,\infty) $. Thus, $ f\circ u $ is a positive superharmonic function which explains the choice of the weight $ g $ in the theorem above. 
	
	Moreover, since the function $ t\mapsto t^{\al} $ is concave and increasing for $ 0\leq \al \leq 1 $ on $ [0,\infty) $ the function $ u^{\al} $ is superharmonic. Hence, the theorem above can be applied replacing $ u $ by $ u^{\al} $ which leads to a weight $ \log^{2}(u^{\al/2}+1) $ and a Hardy weight $ w=\Delta (u^{\al/2})/u^{\al/2} $. It is also not hard to check that a corresponding Rellich-type inequality holds with the constant $ \gamma $ taken from the proof independently of $ \alpha $.
\end{rem}

In order to verify the above corollary, we prove the following lemma.

\begin{lemma}  \label{lemma:sqdlog}
The function $\theta:[c, \infty) \to [0, \infty)$, $\theta(t) := \log(t+1)$ is admissible for every $0<\varepsilon <1$ and $ c>0 $ with   admissible constant 
	\[
\gamma( \varepsilon) = \frac{\left( 1- \frac{\log(c\varepsilon + 1)}{\log(c + 1)} \right)^2}{(1-\varepsilon)^2}\, .
	\]
\end{lemma}
\begin{proof}
We define  the function $ \vartheta:(0, \infty)\times(0,\infty) \to (0, \infty) $
\begin{align*}
 \vartheta(a,t) &= 
 \frac{\left( 1-  \frac{\log(at + 1)}{\log(t + 1)}\right)^2}{(1-a)^2}\, , 
 \end{align*}
 for $ a \neq 1 $
 and
 \begin{align*} \vartheta(1,t)&=
 \frac{t^2}{(t+1)^2 \log^2(t+1)}\, .	
\end{align*}
With this definition we have
\begin{align*}
(\log(at + 1)-\log(t + 1))^{2}= \vartheta(a,t)(1-a)^{2}\log^{2}(t+1).
\end{align*}
We show that $ \vartheta $ is monotone decreasing in both parameters and strictly monotone decreasing in the first one. Together with the fact  $  \lim_{a\to 0}\vartheta (a,c)=1$ this yields that $ \theta  $ is admissible with admissible constant $ \gamma=\vartheta(\eps,c) $ as then
\begin{align*}
\vartheta (a,t)\leq \vartheta (\eps,c)<1
\end{align*}
for $ a\ge \eps $ and $ t\ge c $. The monotonicity is proven by two claims.
\medskip

\emph{Claim 1:} We have $\partial_{t}\vartheta\leq 0$.

\emph{Proof of Claim~1:} First of all observe that $(1-a)^{2}\vartheta(a,\cdot) = 0$ at $ a=1 $ so there is nothing to show for $a=1$.  For $ a\neq 1 $, one computes
	\begin{equation*} 
	\partial_{t}\vartheta(a,t) = \frac{2}{(1-a)^{2}}  \left(\! 1- \frac{\log(at+1)}{\log(t+1)} \!\right)\!\!  \left(\! \frac{\log(at+1)}{(t+1) \log^2(t+1)} - \frac{a}{(at+1)\log(t+1)} \right).
	\end{equation*}
	
	We distinguish the cases $a > 1$ and $a < 1$. 
	So let $a > 1$. Clearly, in this case, the second factor in the above expression is less or equal than $0$. So, we need to show that the third factor is greater or equal than $0$. Indeed, the third factor can be expressed as
	\begin{align*}
	\frac{(at+1)\log(at+1)-a(t+1)\log(t+1)}{(t+1)(at+1)\log^{2}(t+1)}=\frac{f(a)-af(1)}{(t+1)(at+1)\log^{2}(t+1)}\, ,
	\end{align*}
	with $ f(a)=(at+1)\log(at+1) $. Hence, it suffices to show $ f(a)-af(1)\ge 0 $ for $ a>1$. We use the basic estimate $$  f(a)\ge (a-1)\inf_{b>1}f'(b)+f(1),  $$  together with $$  f'(b)=t(1+\log(bt+1))\ge t(1+\log(t+1)),  $$ to estimate
	\begin{align*}
	f(a)-af(1)&\ge (a-1)t(1+\log(t+1)) +(1-a)(t+1)(\log(t+1))\\
	&=(a-1)(t-\log (t+1))\ge0.
	\end{align*}
This shows $ \partial_{t}\vartheta (a,t)\ge0 $ for $ a>1 $.

The proof for $ a<1 $ is similar. Indeed, the second term of $ \partial_{t}\vartheta(a,t) $ is positive for $ a<1 $. The third term can be seen to be negative by the corresponding upper bound on $ f(a)-af(1)\leq (a-1)(t-\log(t+1))\leq 0 $ for $ a<1 $. This finishes the proof of Claim~1.

\medskip

Next, we show that $ \vartheta(\cdot,t) $ is strictly monotonically decreasing for fixed $ t $.
	
	\medskip
	
\emph{Claim 2:} We have	 $\partial_{a}\vartheta<0$. 

\emph{Proof Claim~2:}  One directly computes the derivative of $ \vartheta  $ for $ a\neq 1 $
\begin{align*}
\partial_{a}\vartheta(a,t) \!=\! 
\frac{2\big( \log(t \!+\! 1) - \log(at\!+ \!1) \big)  \big( t(a\!-\!1) + (at \!+\! 1)(\log(t \!+\! 1) - 
	\log(at \!+\! 1)) \big)}{(1-a)^3(at + 1) \log^2(t + 1)}, 			
\end{align*}
	and
\begin{align*}
\partial_{a}\vartheta(1,t)=- \frac{t^3}{(t+1)^3 \log^2(t+1)}\,.
\end{align*}
Clearly, $\partial_{a}\vartheta(1,t)<0 $. Suppose that  $a > 1$. Then, by the mean value theorem  we obtain
	\[
	{ \log(t + 1)}-\log(at + 1) < -\frac{t(a-1) }{(at + 1)}\, .
	\]
Hence, the first factor of the enumerator in the expression of $ \partial_{a}\vartheta(a,t) $ is negative and the second factor is negative as well since
\begin{align*}
t(a-1) + (at + 1)(\log(t + 1) &-  \log(at + 1)) < t(a-1)\left(1- \frac{ (at + 1)}{(at + 1)}\right)= 0
\end{align*}
for $ a>1 $. Moreover, the denominator is negative. The case $a < 1$ is proved in a similar way. 
	We conclude that $\vartheta(\cdot,t) $ is strictly monotonically decreasing. This proves the Claim~2.
	
	 By the discussion above the proof is complete.
\end{proof}

	\begin{proof}[Proof of Corollary~\ref{thm:Rellichboundedlog}]
	 We denote the positive lower bound of $u$ by $c$. It follows from Lemma~\ref{lemma:sqdlog} that $\theta$ given by $\theta(t)=\log(t+1) $ is admissible on the interval  $[c^{1/2}, \infty)$. Hence, we can apply Theorem~\ref{thm:RellichGENGLOB} in order to obtain the Rellich-type inequality.  
	\end{proof}

\begin{rem}
		Unlike the case where we dealt with the admissible function 
	$\theta(t) = t^{\alpha}$, for $\theta(t) = \log(t+1)$ there is another parameter $c$ (the lower bound of the superharmonic function $u$ in question) besides $\varepsilon$ determining the admissibility constant $\gamma$, see Lemma~\ref{lemma:sqdlog}.  
\end{rem}

As in the previous section, we apply the above theorem to the standard line graph in $\N$. 

\begin{eg}\label{eg:linegraphroot2}
We return to the Example~\ref{eg:linegraphroot} for the standard line graph on $\N_{ 0}$ with the  Hardy weight $w$ given by $ w=
|\nabla (n^{1/2})|^{2}/n $ arising from the harmonic function $n(k)=k$. With $D = 2$ and $c= 1$, we define  
 \begin{eqnarray*}
 	\gamma := \frac{\Big( 1 - \frac{\log(2^{-1/2} + 1)}{\log 2} \Big)^2}{(1-2^{-1/2})^2} \approx 0.6083.
 	\end{eqnarray*}
  Hence, Corollary~\ref{thm:Rellichboundedlog} with $\theta(t)=\log(t+1) $ leads to the Rellich-type inequality
 \[
 \left\| 1_{\ph}(\Delta \varphi)  \right\|_{\frac{\log^{2}(n^{1/2} + 1)}{w}} \geq \left( 1 -  \gamma \right) \left\| \varphi\right\|_{\log^{2}(n^{1/2} + 1) w},
 \]
 for all $\varphi \in C_c(\N_{ 0})$ with $\varphi(0) =0$,
 where $\log(n^{1/2} + 1)$ takes the role of $g^{1/2} = \theta\circ\big( n^{1/2}\big)$.
\end{eg}

\begin{eg} We revisit the quadrant of Example~\ref{eg:quadrant}. So, we consider the subgraph  $ \N^{d}_{0} $ of $ \Z^{d} $,  the set $ K:=\Z^{d}\setminus\N^{d} $, the positive harmonic function  $ x: \N^{d}_{0}\to[0,\infty) $ given by $ x(k)=k_{1}\cdot\ldots\cdot k_{d} $, and the Hardy weight $ w_{d} $ arising from the positive supersolution $ x^{1/2} $. Hence, 
	$ w_{d}(k)  =w(k_{1})+\ldots +w(k_{d}) $,
	where $ w $ is the Hardy weight of $ \N_{0} $ from Example~\ref{eg:linegraphroot2} above. Thus, we get a Rellich inequality
	\begin{align*}
	\left\|1_{\ph}\Delta \ph\right\|_{\frac{\log^{2}(x^{1/2}+1)}{w_{d}}}\ge \left(1-\gamma\right)	\left\| \ph\right\|_{\log^{2}(x^{1/2}+1)w_{d}}
	\end{align*}
for all $\varphi \in C_c(\N^{d})$ with $ \gamma $ chosen according to Corollary~\ref{thm:Rellichboundedlog}.
\end{eg}

\section{Schr\"odinger operators on discrete measure spaces}\label{sec-Schr} 
In the present section we extend the Rellich inequalities obtained in the previous sections for the Laplacian,  to the more general setting of Schr\"odinger operators on discrete measure spaces. Specifically, we allow for a measure on $ X $ and we add a potential to $ \Delta $. It turns out that by the virtue of the ground state transform this case reduces to the case of the Laplacian above.

We consider a graph $ b $ over $ X $ equipped with a measure $ m $ of full support, and let $ q:X\to \R $  be a given potential. The Schr\"odinger operator $ H=\frac{1}{m}\Delta +q$ acts on $ \mathcal{F}(X) $ as 
\begin{align*}
Hf(x):=\frac{1}{m(x)}\sum_{y\in X}b(x,y)(f(x)-f(y))+q(x)f(x).
\end{align*}
With this definition we also have a  Green's formula (confer \cite[Lemma~4.7]{HK})
for $\ph \in C_c(X)$ and $f \in \mathcal{F}(X)$ 
\begin{equation*}
\frac{1}{2}\sum_{X\times X}b\nabla \ph \nabla f+\sum_{X}mq\ph f=	\sum_{X}m \ph H f = 	\sum_{X} mf H \ph .
\end{equation*}
 We say that a  {\em Hardy inequality} with Hardy weight $ w $ is satisfied for $ H $ on $\ell^{2} (X,m) $ if  for all  $\ph\in C_{c}(X)$
\begin{align*}
\sum_{X}(|\nabla \ph|^{2}+mq\ph^{2})\geq	\| \ph\|_{mw}^2.
\end{align*}
 We prove the following theorem.

\begin{thm}[Abstract Rellich-type inequality for Schr\"odinger operators] \label{thm:abstractRellichSchrodinger}
	Let $b$ be a graph over $(X,m)$, and  let $ q $ be a  potential. Suppose there is a strictly positive Hardy weight $ w $ for $ H $ on   $\ell^{2} (X,m) $.
	If there is  a strictly positive function    $g \in \mathcal{F}(X)$ and $0 < \gamma < 1$ such that $g$ satisfies the  eikonal inequality
	\[
	\frac{|\nabla \left(g^{1/2}\right)|^2}{g}\leq \gamma  wm  \qquad \mbox{on }  X,
	\]	
	then the following Rellich inequality holds for all $  \varphi \in C_c(X)  $
	\[
	\left \|  1_{\ph} H \varphi \right\|_{\frac{g}{w}m} \geq  (1-\gamma)  \|   \varphi \|_{gwm}.
	\]
\end{thm}
\begin{rem}
	In contrast to  Theorem~\ref{thm:abstractRellich2} and Corollary~\ref{cor:abstractRellich}, we do not consider subsets $ Y\subseteq X $ in the theorem above. The reason is that the case of subsets is implicitly included. Assume that  the assumptions of the theorem are fulfilled for a Schr\"odinger operator $ H=\frac{1}{m}\Delta +q $ only for a subset $ Y \subseteq X$. Let $ b_{Y} $ be the restriction $ b $ of to $ Y\times Y $,  and denote the corresponding Laplacian of $ b_{Y} $ on $ \mathcal{F}(Y) $ by $ \Delta_{Y} $. Let  $ q_{Y} $ and $ m_{Y} $ be the restrictions  of $ q $ and $ m $ to $ Y $. Furthermore, let $ q^{D}_{Y}:Y\to [0,\infty) $ be given by
	\begin{align*}
	q^{D}_{Y}(x):=\sum_{y\in X\setminus Y}b(x,y),
	\end{align*}
	and define the Schr\"odinger operator $ H_{Y} =\frac{1}{m_{Y}}\Delta_{Y}+q_{Y}+q^{D}_{Y}$ on $ \mathcal{F}(Y) $ by
	\begin{align*}
	H_{Y}f(x)=\frac{1}{m_{Y}(x)}\sum_{y\in Y}b_{Y}(x,y)(f(x)-f(y))+(q_{Y}+q^{D}_{Y})(y)f(x).
	\end{align*}
		Then, one can check by a direct computation that 
	\begin{align*}
	H=H_{Y}\qquad \mbox{on }\mathcal{F}(Y).
	\end{align*}
	As the theorem above applies to $ H_{Y} $ it follows directly for $ H $ on $ Y $ as well. In the proof of Theorem~\ref{thm:abstractRellichSchrodinger} we use the ``reverse'' construction.
\end{rem}

In the next proposition, we show that we can deduce a Rellich inequality for operators $ H=\frac{1}{m}\Delta $ from the more general version of the abstract Rellich inequality, Theorem~\ref{thm:abstractRellich2}.

\begin{pro}[$ q= 0 $]\label{pro:q=0}
	Let $b$ be a graph over $(X,m)$, and let $ q =0$. Suppose that there exists a strictly positive  Hardy weight $ w $ for $ H $ on   $\ell^{2} (Y,m) $ with $ Y\subseteq X $, and  a strictly positive function    $g \in \mathcal{F}(Y)$ and $0 < \gamma < 1$ such that $g$ satisfies the weak eikonal inequality
for some (respectively, all) $ \ph\in C_{c}(Y) $
\[
\frac{1}{2}	\sum_{X\times X}b (\ph\otimes \ph)\left(\nabla \left(g^{1/2}\right)\right)^2\leq \gamma\| \ph\|^{2}_{ gwm}.
\]	
	Then, for such a (respectively, all) $ \ph\in C_{c}(Y) $,
	\[
	\left \|  1_{\ph} H \varphi \right\|_{\frac{g}{w}m} \geq  (1-\gamma)  \|   \varphi \|_{gwm}.
	\]
\end{pro}
\begin{proof}The aim is to reduce the case above to the setting of Theorem~\ref{thm:abstractRellich2}.
If $ w $ is a Hardy weight for $ H=\frac{1}{m}\Delta$ on $ \ell^{2}(Y,m) $, then the function $ w'= wm $ is a Hardy weight for $ \Delta $ on $ \ell^{2}(Y)=\ell^{2}(Y,1) $, $ Y\subseteq X $. 	Hence, we get by the more general version of abstract Rellich inequality, Theorem~\ref{thm:abstractRellich2},
	\begin{align*}
		\| 1_{\ph} \tfrac{1}{m}\Delta\ph\|_{\frac{g}{w}m}=\| 1_{\ph} \Delta\ph\|_{\frac{g}{w'}}\ge (1-\gamma)\|\ph\|_{gw'}=(1-\gamma)\|\ph\|_{gwm}.
	\end{align*}
	This finishes the proof.
\end{proof}

\begin{proof}[Proof of Theorem~\ref{thm:abstractRellichSchrodinger}]
	By the theorem's assumptions $ H $ satisfies a Hardy inequality with a positive Hardy weight and by the Allegretto-Piepenbrink theorem, \cite[Theorem~4.2]{KePiPo1} there exists a strictly positive  function $ u $ such that $ Hu\ge0  $ on $ X $. By the ground state transform \cite[Section~4.2]{KePiPo1} we have for all $ \ph\in C_{c}(X) $
	\begin{align*} 	\frac{1}{2}\sum_{X\times X}b (u\otimes u)(\nabla \ph)^{2}+\sum_{X}u^{2}m\left(\frac{1}{u}Hu\right) \ph^{2}	=\sum_{X}(|\nabla (u\ph)|^{2}+mq(u\ph)^{2}).\end{align*}
	Hence, the Schrödinger operator with respect to the graph $ b (u\otimes u) $ and the potential $ \frac{1}{u}Hu $ satisfies a Hardy inequality on $ \ell^{2}(X,u^{2}m) $ with respect to the Hardy weight $ w $. We next construct a supergraph of this graph on the set $ X_{\infty}=X\cup\{\infty\} $ where $ \infty $ is a new vertex. Let $b_{\infty}:X_{\infty}\times X_{\infty}\to[0,\infty)  $ be the graph given by $ b_{\infty}= b(u\otimes u)$  on  $X\times X $,
\begin{align*}b_{\infty}(x,\infty)=b_{\infty}(\infty,x)=mu(x)(Hu)(x)\end{align*}
and $ b_{\infty}(\infty,\infty)=0 $.
	Note that $ \sum_{x\in X}b(\infty,x) $ is not necessarily finite,  but this does not matter as we will only consider functions which vanish at the vertex $ \infty $. With this construction we have using the Hardy inequality for $ H $  and all $ \ph\in C_{c}(X) $
	\begin{align*}	\frac{1}{2}\sum_{X_{\infty}\times X_{\infty}}b_{\infty}(\nabla \ph)^{2}= \sum_{X}(|\nabla (u\ph)|^{2}+mq(u\ph)^{2})\ge\|u\ph\|_{wm}^{2}= \|\ph\|_{u^{2}wm}^{2}.
	\end{align*}
	Thus, $ w $ is a Hardy weight on $\ell^{2}( X,u^{2}m)$ for the operator $ H_{\infty} $ associated to the graph $ b_{\infty} $ over $( X_{\infty},m_{\infty}) $, where we extend $ m $ to $ \{\infty\} $ arbitrarily and the potential $ q_{\infty}=0 $.
	 Furthermore, from the assumption $  {|\nabla g^{1/2}|^2}\leq \gamma  gwm$ 
	and from Young's inequality, $ 2((u\ph)\otimes (u\ph))(x,y) \leq (u\ph)^2(x)+(u\ph)^2(y)$, we deduce for $ \ph\in C_{c}(X) $
	\begin{align*}
\frac{1}{2}	\sum_{X\times X}b_{\infty}(\ph\otimes\ph)(\nabla g^{1/2})^{2}&=	\frac{1}{2}\sum_{X\times X}b\big((u\ph)\otimes (u\ph)\big)(\nabla  g^{1/2})^{2}\\
	&\leq \sum_{X} (u\ph)^{2}|\nabla g^{1/2}|^{2}\\
	&\leq\gamma \|\ph\|^{2}_{ gwu^{2}m}.
	\end{align*}
 Hence, we can apply Proposition~\ref{pro:q=0} with $ X\subseteq X_{\infty} $ playing the role of $ Y\subseteq X $ and $ H_{\infty} $ playing the role of $H  $. Specifically,  we obtain
for any $ \ph= u\psi\in C_c(X) $
	\begin{equation*}
\left \|  1_{\ph} H \ph \right\|_{\frac{g}{w}m} =	\left \| 1_{\psi}  H_{\infty} \psi  \right\|_{\frac{g}{w}u^2m} \geq  (1-\gamma)  \|   \psi \|_{gwu^{2}m}=(1-\gamma)\|   \varphi \|_{gwm}.  
	\end{equation*}
This finishes the proof.
\end{proof}

\section{Application to solutions of the equation $ Hu=f $}\label{sec_applications}

In this section we study solutions $ u\in \mathcal{F} $ of the equation $$ Hu=f  ,$$
where $ f $ is a function and $ H=\frac{1}{m}\Delta+q $ is a Schr\"odinger operator associated to a graph $ b $ over $ (X,m) $ and a potential $ q $.  We call a Schr\"odinger operator $ H $ {\em positive} on $(X,m)$ if for all $ \ph \in C_c(X) $
\begin{align*}
\sum_{X}m\ph H\ph=\sum_{X}(|\nabla \ph|^{2}+mq\ph^{2})\ge0,
\end{align*}
where the first equality is just Green's formula. We call the graph $ b $ \emph{connected} if for distinct $ x,y\in X $ there are $ x_{0} , \ldots, x_{n}\in X $ such that $ x=x_{0} \sim  \ldots\sim x_{n}=y $. 
Furthermore, we call a subset $ K $ connected if $ b\vert _{K\times K} $ is a connected graph.

In general, solutions of the equation $Hu=f$ do not need to exist. 
However, if $H$ is a positive Schr\"odinger operator on a connected, locally finite graph, then $H$ is surjective \cite[Theorem~2.2]{KS}.  Using the results of the previous section, we show the existence of solutions of the above equation on not necessarily locally finite graphs  under the assumption of the validity of a Rellich inequality.

\begin{thm}\label{thm-sol}
Let $H$ be a positive Schr\"odinger  operator for a connected graph over an infinite set $X$. Let $ \mu,\mu'$ be  measures of full support over $ X $  such that the following Rellich inequality holds
\begin{align*}
\|1_{\ph}H\ph\|_{\mu'}\ge \|\ph\|_{\mu} 
\end{align*}
for all $ \ph\in C_{c}(X) $. 
Then, for any $ f\in\ell^{2}(X,\mu') $, there exists a solution $ u \in \mathcal{F}(X)$ to the equation $ Hu=f $ such that
\begin{align*}
 \|u\|_{\mu} \leq \|f\|_{\mu'}. 
\end{align*}	
\end{thm}
\begin{rem}
	If there exists a nontrivial harmonic function $ h $ for $ H $, then obviously, solutions to $ Hu=f $ are not unique. In this situation also the inequality $\|v\|_{\mu}\leq \|f\|_{\mu'}$ is not satisfied for solutions $v$ of the form $ v=u+\lm h $, at least for certain  $\lambda\in \R$.
\end{rem}

\begin{proof}[Proof of Theorem~\ref{thm-sol}]
Assume first that $ f\in \ell^{2}(X,\mu') $ satisfies $ f\ge0 $, and let $ (K_{k})$ be an exhaustion of $ X $ with connected finite sets $K_k$, ${k\in \N} $. We denote by $ f_{k}=f1_{K_{k}} $. Then, by the positivity of $ H $, there exists solutions $ u_{k}\ge0 $ to 
\begin{align*}
H u_{k}=f_{k}\qquad \mbox{on }K_{k},
\end{align*}
satisfying $ u_{k}\leq u_{k+1} $ for $k\geq 1$, see Lemma~5.14, Lemma~5.15, and Theorem~5.16 in \cite{KePiPo1}. By the Rellich inequality, we infer that
\begin{align*}
 \|u_{k}\|_{\mu}\leq \|1_{u_{k}}Hu_{k}\|_{\mu'}=\|f_{k}\|_{\mu'}\leq \|f\|_{\mu'}\,.
\end{align*}
Thus, $ (u_{k}) $ converges monotonously to a function $ u\ge0 $ which satisfies 
\begin{align*}
\|u\|_{\mu}\leq \|f\|_{\mu'}.
\end{align*}
Furthermore, by the monotone convergence we also deduce for all $  x\in X $ 
\begin{align*}
Hu(x)=\lim_{k\to\infty}Hu_{k}(x)=\lim_{k\to\infty}f_{k}(x)=f(x).
\end{align*}
This yields the statement of the theorem for $ f\in \ell^{2}(X,\mu') $ such that $ f\ge0 $.

\medskip

For general $ f\in \ell^{2}(X,\mu') $ we decompose $ f=f_{+}-f_{-} $ into its positive and negative parts, and let $u^{(\pm)}$ be the positive solutions of the equations $Hu^{(\pm)}=f_\pm$ obtained above.  Hence, $u= u^{(+)}-u^{(-)} $ satisfies  
$$ Hu=Hu^{(+)}-Hu^{(-)}=f_{+}-f_{-}=f .$$   Since $ u^{2}=(u^{(+)})^2 +(u^{(-)})^2-2u^{(+)}u^{(-)}\leq (u^{(+)})^2 +(u^{(-)})^2$, we have
\begin{align*}
\|u\|_{\mu}^{2}\leq \|u^{(+)}\|_{\mu}^{2}+\|u^{(-)}\|_{\mu}^{2}\leq \|f_{+}\|_{\mu'}^{2}+\|f_{-}\|_{\mu'}^{2}=\|f\|_{\mu'}^{2}.
\end{align*}
This finishes the proof.
\end{proof}

Finally, we present a corollary which is a direct consequence of Theorem~\ref{thm:abstractRellichSchrodinger} and Theorem~\ref{thm-sol}.
\begin{thm}
	Consider  a graph $b$ over $(X,m)$, let $ q $ be a potential.  Let $w$ be a strictly positive  Hardy weight for $ H $ on   $\ell^{2} (X,m) $.
Assume there is  a strictly positive function    $g \in \mathcal{F}(X)$ and $0 < \gamma < 1$ such that $g$ satisfies in $X$  the   eikonal inequality
	\[
	\frac{|\nabla \left(g^{1/2}\right)|^2}{g}\leq \gamma  wm .
	\]	
	Then, for all $ f\in \ell^{2}(X,{g}m/w) $, there is a solution $ u $ to the equation $ Hu=f $ which satisfies
	\[
	  \|   u \|_{gwm} \leq (1-\gamma)^{-1} \left \| f\right\|_{\frac{g}{w}m} .
	\]
\end{thm}
\bibliographystyle{alpha}
\bibliography{literature}

\end{document}